\documentclass{article}
\usepackage[utf8]{inputenc}
\usepackage{amsmath}
\title{loop grassmaniann}
\author{dongmouren }
\date{February 2015}
\usepackage[utf8]{inputenc}
\usepackage{amsmath}
\usepackage{amssymb}
\usepackage{amsthm}
\usepackage{youngtab}

\usepackage[english]{babel}

\newtheorem{theorem}{Theorem}
\newtheorem{lemma}{Lemma}
\newtheorem{sublemma}{Sublemma}
\newtheorem{prop}{Proposition}
\newtheorem{conj}{Conjecture}
\newtheorem{defi}{Definition}
\newtheorem{cor}{Corollary}
\title{loop grassmaniann}
\author{dongmouren }
\date{February 2015}
\newcommand{\fg}{\mathfrak{g}}
\newcommand{\fn}{\mathfrak{n}}
\newcommand{\cI}{\mathcal{I}}
\newcommand{\cJ}{\mathcal{J}}
\newcommand{\cO}{\mathcal{O}}

\newcommand{\cG}{\mathcal{G}}
\newcommand{\cL}{\mathcal{L}}

\newcommand{\cK}{\mathcal{K}}

\newcommand{\cR}{\mathcal{R}}
\newcommand{\ZZ}{\ensuremath{\mathbb{Z}}}
\newcommand{\NN}{\ensuremath{\mathbb{N}}}
\newcommand{\CC}{\ensuremath{\mathbb{C}}}
\begin{document}

\title{A relation between Mirković-Vilonen cycles and modules over preprojective algebra of Dynkin quiver of type ADE}
\author{Zhijie Dong }
\date{}
\maketitle
\begin{abstract}
The irreducible components of the variety of all modules over the preprojective algebra and MV cycles both index bases of the universal enveloping algebra of the positive part of a semisimple Lie algebra canonically. To relate these two objects Baumann and Kamnitzer associate a cycle in the affine Grassmannian for a given module. It is conjectured that the ring of functions of the T-fixed point subscheme of the associated cycle is isomorphic to the cohomology ring of the quiver Grassmannian of the module. I give a proof of part of this conjecture. Given this conjecture, I give a proof of the reduceness conjecture in \cite{kamnitzer2016reducedness}.

\end{abstract}

\section{Introduction}

Let $\mathfrak{g}$ be simply-laced semisimple finite dimensional complex Lie algebra.
There are several modern constructions of irreducible representation of $\mathfrak{g}$. In this paper we consider two models which realize the crystal of the positive part $U(\fn)$ of $U(\fg)$. One is by Mirkovic-Vilonen (MV) cycles and the other is by the irreducible compotents of Lusztig's nilpotent variety $\Lambda$ of the preprojective algebra of the quiver $Q$ corresponding to $\fg$.

Baumann and Kamnitzer \cite{MR2892443} studied the relations between $\Lambda$ and MV polytopes. They associate an MV polytope $P(M)$ to a generic module $M$ and construct a bijection between the set of irreducible components of $\Lambda$ and MV polytopes compatible with respect to crystal structures.
Since MV polytopes are in bijection with MV cycles, Kamnitzer and Knutson launched a program towards geometric construction of the MV cycle $X(M)$ in terms of a module $M$ over the preprojective algebra.

Here we consider a version by Kamnizter, Knutson and Mirkovic: conjecturally, the ring of functions $\cO(X(M)^{T})$ on the $T$ fixed point subscheme of the cycle $X(M)$ associted to $M$, is isomorphic to $H^{*}(Gr^{\Pi}(M))$, the cohomology ring of the quiver Qrassmannian of $M$.
In this paper I will construct a map from $\cO(X(M)^{T})$ to $H^{*}(Gr^{\Pi}(M))$ and prove it is isomorphism for the case when $M$ is a representation of $Q$.

In section two I recall the definition of MV cycles, quivers, preprojective algebra and Lusztig's nilpotent variety and state the conjecture precisely.

In chapter three I describe the ring of functions on the $T$ fixed point subscheme of the intersection of closures of certain semi-infinite orbits (which is called "cycle" in this paper). A particular case of these intersections is a scheme theoretic version of MV cycles. We realize these cycles as the loop Grassmannian with a certain condition Y.

In section four I construct the map $\Psi$ from $\cO(X(M)^{T})$ to $H^{*}(Gr^{\Pi}(M))$. Here, $\Psi$ maps certain generators of $\cO(X(M)^{T})$ to Chern classes of tautological bundles over $Gr^{\Pi}(M)$. So we need to check that the Chern classes satisfy the relations of generators of $\cO(X(M)^{T})$. We reduce this problem to a simple $SL_3$ case. In this case we have a torus action on $Gr^{\Pi}(M)$ so we could use localization in equivariant cohomology theory (GKM theory).

In chapter five I will prove $\Psi$ is an isomorphism in the case when $M$ is a representation of $Q$ of type A. 

In chapter six I will state some consequences given the conjecture
(one of which is the reduceness conjecture).

%and describe some possible approach in general cases.
%(to be written).
This is a piece of a big project to relate $\cG(G)$ and the quiver $Q$, see \cite{localivan}, and more recently \cite{ivan2}.
\section{Statement of the conjecture}
\subsection{Notation}
Let $G$ be a simply-laced semisimple group over complex numbers. Let I be the set of vertices in the Dynkin diagram of $G$. In this paper I will work over base field $k=\CC$. We fix a Cartan subgroup $T$ of $G$ and a Borel subgroup $B\subset G$. Denote by $N$ the unipotent radical of $B$. Let $\varpi_i,i\in I$ be the fundamental weights. Let $X_*, X^{*}$ be the cocharacter, character lattice and $\langle\ ,\ \rangle$ be the pairing between them. Let W be the Weyl group. Let $e$ and $w_0$ be the unit and the longest element in $W$. Let $\alpha_i$ and $\check{\alpha}_i$ be simple roots and coroots. Let $\Gamma=\{w\varpi_i,w\in W ,i\in I\}$. $\Gamma$ is called the set of chamber weights.\\

Let $d$ be the formal disc and $d^{*}$ be the punctured formal disc. The ring of formal Taylor series is the ring of functions on the formal disc, $ \cO=\{\sum_{n\geq 0}a_{n}t^{n}\}$. The ring of formal Laurent series is the ring of functions on the punctured formal disc, $\cK=\{\sum_{n\geq {n_0}}a_{n}t^n\}$. \\

For $X$ a variety, let Irr($X$) be the set of irreducible components of $X$.

\subsection{MV cycles and polytopes}
For a group $G$, let $G_{\cK}$ be the loop group of $G$ and $G_{\cO}$ the disc group of $G$. 
We define loop grassmannian $\cG(G)$ as the left quotient $G_{\cO} \setminus G_{\cK}$ and  view $\cG(G)$ as an ind-scheme \cite{beilinson1991quantization}, \cite{zhu}.
An MV cycle is a certain finite dimentional 
subscheme in $\cG(G)$. For a cocharacter $\lambda \in X_*(T)$, we denote the point it determines in $\cG(G)$ by $L_\lambda$.
For $w\in W$, let $N^{w}=wNw^{-1}$.
Define $
S^{w}_\lambda =L_\lambda N^{w}_\cK. $
This orbit is an ind-subscheme of $\cG(G)$ and is called semi-infinite orbit since it is of infinite dimension and codimension in $\cG(G)$.\\
An irreducible component of $\overline {S_0^{e} \bigcap S^{w_0}_\lambda}$ is called an MV cycle of weight $\lambda$. Kamnizter \cite{kamnitzer2005} describes them as follows:
\begin{theorem}[\cite{kamnitzer2005}]
Given a collection of integers $(M_\gamma)_{\gamma \in \Gamma}$, if it satisfies edge inequalities, and certain tropical relations, put  $\lambda_{w}=\sum_i M_{w \varpi_i} w \check{\alpha_i}.$\\Then$
\overline {\bigcap_{w\in W}  S^{w}_{\lambda_{w}}} $ is an MV cycle, and each MV cycle arises from this way for the unique data $(M_{\gamma})$.

\end{theorem}
The data $(M_\gamma)_{\gamma \in \Gamma}$ determines a pseudo-Weyl polytope. It is called an MV polytope if the corresponding cycle $\overline {\bigcap_{w\in W}  S^{w}_{\lambda_{w}}} $ is an MV cycle.
MV polytopes are in bijection with MV cycles.
Using this description, Kamnizter \cite{kcrystal} reconstruct the crystal structure for MV cycles.

\begin{prop}[\cite{kcrystal}]
MV polytopes have a crystal structure isomorphic to $B(\infty)$.
\end{prop}
\subsection{Objects on the quiver side}
Let $Q=\{I,E\}$ be a Dynkin quiver of type ADE, 
where I is the set of vertices and E is the set of edges.
We double the edge set E by adding all the opposite edges. Let $E^*=\{a^* | a\in E\}$
where for $a:i\xrightarrow[]{}j, a^*=j\xrightarrow[]{}i$, also we define $s(a)=i,t(a)=j$. Define $\epsilon(a)=1 $ when $a\in E$, $\epsilon(a)=-1$, when $a\in E^*$.
Let $H=E\bigsqcup E^*$ and $\overline{Q}=\{I, H\}$.
The preprojective algebra $\Pi$ of $Q$ is defined as  quotient of the path algebra by a certain ideal:

$$
\Pi_{Q}= k\overline{Q} /<\sum_{a\in H} \epsilon(a)aa^*>.\footnote{Since most time we fix Q so Q is omitted when there is no confusion.}$$

A $\Pi_{Q}-$module is the data of an $I$ graded vector space $\bigoplus_{i\in I} M_i$ and linear maps $\phi_a : M_{s(a)} \xrightarrow[]{} M_{t(a)}$ for each $a\in H$ satisfying the preprojective relations $\sum_{a\in H,t(a)=i} \epsilon(a)\phi_{a}\phi_{a*}=0$.

Given a dimension vector $d\in \NN^I$, define $\Lambda(d)$ to be the variety of all representations of $\Pi$ on M for $M_i=k^{d_{i}}$.
\begin{prop}[\cite{lusztig1990canonical}, \cite{MR2892443}]
Irr$(\Lambda)$ has a crystal structure isomorphic to $B(\infty)$.
\end{prop}

\subsection{A conjectural relation between MV cycles and modules over the preprojective algebra}
Baumann and Kamnitzer found an isomorphism between the crystal structure of Irr$(\Lambda)$ and MV polytopes. 
For each $\gamma\in \Gamma$, they define constructible funtion $D_{\gamma}: \Lambda(d) \xrightarrow[]{} \ZZ_{
{\geq0}}$\footnote{$\Lambda(d)$ and $D_{\gamma}$ do not depend on the direction of the edges in E.}. 
For any $M\in\Lambda(d)$, the collection $ (D_{\gamma})_{\gamma\in \Gamma}$ satisfies certain edge inequalities hence determines a polytope which we denote by $P(M)$.

\begin{theorem}[\cite{MR2892443}]
When $M$ is generic, $P(M)$ is an MV-polytope and for $d=(d_i)_{i\in I}$ this gives a map from Irr($\Lambda(d))$ to the set of MV polytopes of weight $\sum_{i\in I} d_i\alpha_i$. This map is a bijection compatible with the crystal structures.

\end{theorem}
We have MV-cycles (in bijection with MV-polytopes) as the geometric object on the loop Grassmannian side.
In order to upgrade the relations geometrically,
Kamnitzer-Knutson consider the quiver Grassmannian on the quiver side.

The quiver Grassmannian $Gr^{\Pi}(M)$ of a $\Pi$-module $M$ is defined as the moduli of submodules of $M$.

It is a subscheme of the moduli of $k$-vector subspaces of $M$ which is product of usual grassmannian $\prod_{i\in I} Gr(M_i)$. Here we will only consider $Gr^{\Pi}(M)$ with its reduced structure, and actually just as a topological space.
As the case of usual grassmannian, the
quiver Grassmannian $Gr^{\Pi}(M)$ is disjoint union of Grassmannians of different dimension vectors. Denote $Gr_{e}^{\Pi}(M)$ by the moduli of submodule N of M of dimension vector $e$, we have
$Gr_{e}^{\Pi}(M)\subset \prod_{i\in I} Gr_{e_i}(M_i)$.\\

Given a module $M\in \Lambda(d)$, form the subscheme\footnote{We will call it cycle in this paper.} $X(M)$=$\bigcap_{w\in W} \overline{S^{w}_{\lambda_{w}}}$, where
 $\lambda_w=\sum_{i\in I}-D_{-w\varpi_{i}}(M)w\check{\alpha{_i}}$. 
T acts on $S^{w}_{\lambda_w}$ by multiplication, hence it also acts on the closure and the intersection $X(M)$.\\

\begin{conj}
The ring of functions on the $T$-fixed point subscheme of $X(M)$ is isomorphic to the cohomology ring of the quiver grassmannian of $M$
$$\cO (X(M)^{T}) \xrightarrow[\sim]{\Psi} H^*(Gr^{\Pi}(M)).$$
More precisely, $X(M)^{T}$ is disjoint union of finite schemes $X(M)^{T}_{\nu}$ supported at $L_\nu$, $\nu\in X_{*}(T)$ and we can further identify two sides for each connnected component
$$\cO (X(M)_{\nu}^{T}) \xrightarrow[\sim]{\Psi} H^*(Gr_{e}^{\Pi}(M)),
\text{ where } e_i=(\nu, \varpi_i).$$
\end{conj}
 
Remark: we define $X(M)$ as a scheme theoretic intersection of closures while MV-cycles have been defined as varieties (closure of intersections). We notice that $X(M)$ may be reducible even when $P(M)$ is an MV-polytope.
For an example, see the appendix\footnote{Not yet written, I will add this later on.}.
The former certainly contains the latter and a further hope is to relate the latter to some subvariety of the quiver Grassmannian.
%%
%%The rest of this paper will be devoted to the proof of this conjecture.
%%
\section{The T fixed point subscheme of the cycle}
We introduce some notation first.
It is known that the $T$-fixed point subscheme of the loop grassmannian of a reductive group $G$ is the   loop grassmannian of the Cartan $T$ of $G$, i.e.,
$\cG(G)^{T}=\cG(T)$.
We indentify $T$ with $I$ copies of the multiplicative group by
$T\xrightarrow[\sim]{\prod\varpi_i} G_{m}^{I}$
and this gives $\cG(T)\xrightarrow[\sim]{\prod\varpi_i} \cG(G_m)^{I}$.

For $\cG(G_m)$, we have
\begin{align}
&\cG(G_m)=G_m(\cO) \setminus G_m(\cK)\\
&=\{\text{unit} \in \cO\} \setminus \{\text{unit} \in \cK\}\\
&=t^{\ZZ}\cdot K_{-}
\end{align}
where $K_{-}$ is called the negative congruence subgroup (of $G_m$).
The $R$-points of $K_{-}$ can be described as:   $$K_{-}(R)=\{a=(1+a_1t^{-1}+...+a_mt^{-m})|   a_i \text{ is nilpotent in }R \}.$$
We define the degree function from $K_{-}(R)$ to $\ZZ_{\geq}$: deg$(a)=m$ if $a_m\neq 0$.

Then
$(\bigcap \overline{S^{w}_{\lambda_{w}}})^{T}$
is a subscheme of $\cG(G)^{T}\cong (t^{\ZZ}\cdot K_{-})^{|I|}$.
\begin{theorem}
Let $(\lambda_w)_{w\in W}$ be a collection of cocharacters such that $\lambda_v \geq_{w} \lambda_w$\footnote{This notation is used in \cite{kamnitzer2005}, $\lambda_v \geq_{w} \lambda_w$ whenever $w^{-1}\lambda_v \geq w^{-1}\lambda_w$.} for all $w\in W$ in which case we know (\cite{kamnitzer2005}) that $(\lambda_w)_{w\in W}$ determines a pseudo-Weyl polytope. The integers $A_{w\varpi_i}$ are well defined by $A_{w\varpi_i}=(\lambda_w,w\varpi_i)$. The R-points of $
(\bigcap \overline{S^{w}_{\lambda_{w}}})_{\nu}^{T}$ is the subset of R-point of $(t^{\ZZ}\cdot K_{-})^{|I|}$ containing elements 
$(t^{(\nu,\varpi_{i})}a_i) \in \prod (t^\NN \cdot K_-)^{|I|}|$ subject to the degree relations: $$deg(\Pi_{i\in I} a_i^{(\gamma, \check{\alpha_i})}) \leq -A_{\gamma}+\sum (\gamma,\nu)\text{ for all } \gamma \in \Gamma\}.$$

\end{theorem}

\begin{proof}
We define loop grassmannian with a condition $Y$ and list the facts we need. For details, see \cite{ivan3}.
Let $G$ acts on scheme $Y$ and $y$ be a point in $Y$. Denote the stack quotient by $Y/G$.
Then $\cG(G,Y)$ is the moduli of maps of pairs from $(d, d^{*})$ to $(Y/G, y)$.
When $Y$ is a point we recover $\cG(G)$.
In general, $\cG(G,Y)$ is the subfunctor of $\cG(G)$ subject to a certain extension condition:$$
\cG(G,Y)=G_{\cO}\setminus \{g\in G_{\cK}\  |\  d^{*}\xrightarrow[]{g}G\xrightarrow[]{o} Y \text{ extends to $d$}\}, \text{ where } o(g)=gy.$$\\
We can realize semi-infinite orbits and their closures as follows:
\begin{itemize}

\item $\cG(G,G/N)=S_0$, where G acts G/N by left multiplication.
\item $\cG(G,(G/N)^{aff})=\overline{S_0}$, where "aff" means affinization.
\item $\cG(G\times T,(G/N)^{aff})_{red}=\bigsqcup \overline{S_{\lambda}}$, where "red" means the reduced subscheme. Here T acts on $G/N$ by left multiplication with the inverse and this extends to an action on $(G/N)^{aff}$. 
\item $$ \cG(G\times \prod_{w\in W} T_w, \prod_{w\in W} (G/N^w)^{aff})=\bigsqcup_{(\lambda_w)_{w\in W}} (\bigcap_{w\in W} \overline{S^{w}_{\lambda_{w}}})$$, We denote a copy of T corresponding to $w\in W$ by $T_w$.

\end{itemize}
A single cycle $\bigcap_{w\in W}\overline{ S^{w}_{\lambda_{w}}}$ can be written as the fiber product:
$$\bigcap_{w\in W}\overline{ S^{w}_{\lambda_{w}}}=\cG(G\times \prod_{w\in W} T_w, \prod (G/N^w)^{aff}) \times_{\cG(\prod_{w\in W} T_w)} (t^{\lambda_{w}})_{w\in W}.$$

In this fiber product, the morphism for the first factor is the second projection and 
the morphism for the second factor is the inclusion of the single point $t^{\underline{\lambda}}=(t^{\lambda_{w}})_{w\in W}$.\\
For a reductive group G, we have $\cG(G,Y)^{T}=\cG(T,Y)$, where T is the cartan of G.
\\
So, the T fixed point subscheme is \\ 
$$(\bigcap_{w\in W}\overline{ S^{w}_{\lambda_{w}}})^T=\cG(T\times \prod_{w\in W} T_w, \prod (G/N^w)^{aff}) \times_{\cG(\prod_{w\in W} T_w)} t^{\underline{\lambda}}.$$ \\
In terms of the above extension condition, this fiber product is:
$$(\bigcap_{w\in W}\overline{ S^{w}_{\lambda_{w}}})^T
=T(\cO)\setminus \{ g\in T_{\cK}, \text{ such that  } d^{*}\xrightarrow[]{g,t^{\underline{\lambda}} } T\times T^W \xrightarrow[]{}\prod (G/N^w)^{aff} \text {  extends to } d \}$$
This is the $T(\cO)$ quotient of the set of all $g\in T_{\cK}$, such that $$ d^{*}\xrightarrow[]{g,t^{\lambda_{w}} } T\times T_w \xrightarrow[]{} (G/N^w)^{aff} \text {  extends to }  d \text{ for all }w\in W.
$$
For $\gamma\in W\cdot \varpi_i \subset \Gamma$ , we fix weight vectors $v_{\gamma}  \text{ in the weight space} \
(V_{\varpi_i})_{\gamma} $ of $V_{\varpi_i}$. For each $w\in W$, we embed $G/N^w \text{ into } \bigoplus_{i\in I} V_{\varpi_i} \text{ by } g\mapsto (g\cdot  v_{w\varpi_i})_{i\in I}$. Under this embedding, $(G/N^w)^{aff}$ is a closed subscheme in $\bigoplus_{i\in I} V_{\varpi_i}$.
\\

For $g\in T_{\cK}$, $w\in W$, the composition $y_w(g)$ of the map : $$ d^{*}\xrightarrow[]{g,t^{\lambda_{w} }} T\times T_w \xrightarrow[]{} G/N^w \hookrightarrow \bigoplus V_{\varpi_i}$$is
$$y_{w}(g)=(g \cdot(t^{\lambda_{w}})^{-1})\sum_{i\in I} v_{w\varpi_i}= \sum_{i\in I} (w\varpi_i(g \cdot t^{-\lambda_w}))v_{w\varpi_i}.$$\\
This map extends to $d$ when for each $i\in I$, the coefficient of $v_{w\varpi_i}$ is in $\cO$.
The coefficient of $v_{w\varpi_i}$ is
\begin{align}
w\varpi_i(g \cdot t^{-\lambda_w})=w\varpi_i(g)\cdot w\varpi_i (t^{-\lambda_w})=w\varpi_i(g)\cdot t^{-(w\varpi_i,\ \lambda_w)}\notag\\
=w\varpi_i(g) t^{-A_{w\varpi_i}}
=\gamma(g) z^{-A_\gamma} \text{ where } \gamma=w\varpi_i\notag.
\end{align}

It follows that
$$(\bigcap_{w\in W}\overline{ S^{w}_{\lambda_{w}}})^T=T(\cO)\setminus 
\{g \in T(\cK);
\gamma(g) t^{-A_\gamma} \in \cO \text{ for all } \gamma \in \Gamma\}.$$
and the description of the R-points of $
(\bigcap \overline{S^{w}_{\lambda_{w}}})_{\nu}^{T}$ in the theorem follows when we identify $\cG(T)\xrightarrow[]{\prod\varpi_i} \cG(G_m)^{I}=(t^{\ZZ}\cdot K_{-})^{I}$. 
%%Note an element $g\in \cK $ is in $\cO$ when its degree is no less than 0. 
\\

\end{proof}
\subsection{Ring of functions on $
(\bigcap \overline{S^{w}_{\lambda_{w}}})_{\nu}^{T}$}
For an R-point $(t^{(\nu,\varpi_{i})}a_i)_{i\in I}$ of $
(\bigcap \overline{S^{w}_{\lambda_{w}}})_{\nu}^{T}$, let us write $a_i=1+a_{i1}t^{-1}+\cdots+ a_{im}t^{-m}$. When $\gamma=\varpi_i$, the degree inequality is deg$(a_i)\leq (\varpi_i,\nu)-A_{\varpi_i}$.
 We can take the coefficients $a_{ij}$ to be the coordinate functions on $
(\bigcap \overline{S^{w}_{\lambda_{w}}})_{\nu}^{T}$. Since deg$(a_i)\leq (\varpi_i,\nu)-A_{\varpi_i}$
 , there are finitely many $a_{ij}$s which generate the ring of functions on  $\cO((\bigcap \overline{ S^{w}_{\lambda_{w}}})_{\nu}^{T})$.\\
When we take inverse of $a_i$, it is computed in $K_{-}$ as $a_i^{-1}=1+\sum_{s\geq 0}(-1)^{i}(a_{i1}t^{-1}+\cdots+a_{im}t^{-m})^s$ and then expands in the form $\sum_{i}b_{ik} t^{-k}$, where $b_{ik}$ is the coefficient of $t^{-k}$ in $a_i^{-1}$. 
$$deg(\Pi_{i\in I} a_i^{(\gamma, \check{\alpha_i})}) \leq -A_{\gamma}+\sum (\gamma,\nu)\text{ for all } \gamma \in \Gamma.$$
is equivalent to the condition that the coefficient of the term $t^{-1}$ to the power $-A_{\gamma}+\sum (\gamma,\nu)+1$ in
$(\Pi_{i\in I} a_i^{(\gamma, \check{\alpha_i})})$ is 0. These coefficients are polynomials of $a_{ij}$s.
%%For simplicity of presentation, we still denote the inequalities meaning the ideal generated by those coefficient polynomials.\\
%%So we have\\
%%$$\cO((\bigcap \overline{ S^{w}_{\lambda_{w}}})_{\nu}^{T})=k[a_{ij}]/< deg\prod (a_i)^{\gamma,\check{\alpha_i}}\leq (\gamma,\nu)-A_{\gamma},\gamma\in \Gamma >.$$\\
Set $b_i=1+\sum_{k}b_{ik} t^{-k}=a_i^{-1}$, 
add $b_{ij}$'s as generaters and also add the relations $a_{i}b_{i}=1$ for $i\in I$ which eliminate all $b_{ij}$'s.
For $\gamma \in \gamma$, let $\gamma_i=(\gamma, \check{\alpha_i})$.
Denote by $I_\gamma^{+}$ the subset of I containing all i such that $\gamma_i $ is positive and by $I_\gamma^{-}$ containing all i $\gamma_i$ is negative. Set $\gamma^{+}_i=\gamma_i$ when $\gamma_i$ is positive and $\gamma^{-}_i=-\gamma_i$ when $\gamma_i$ negative.

\begin{cor}
The ring of functions on $\cO((\bigcap \overline{ S^{w}_{\lambda_{w}}})_{\nu}^{T})$ is generated by $a_{ij}$'s and $b_{ik}$'s, for $i\in I$. The relations are degree conditions:
$$ deg(\prod_{i\in I_{\gamma}^{+}} 
a_{i}^{\gamma{_i}^{+}}\prod_{i\in I_{\gamma}^{-}} 
b_{i}^{\gamma{_i}^{-}})\leq (\gamma,\nu)-A_{\gamma}$$ for each $\gamma\in \Gamma$ and conditions $a_{i}b_{i}=1$ for each i in I.

\end{cor}
\section{Construction of the map $\Psi$ from functions to cohomology}
\subsection{Map $\Psi$}
For $M\in \Lambda(d)$, to apply corollary 1 to $X(M)$, we set $A_\gamma=-D_{-\gamma}(M)$. Then
 $$\cO(X(M)_{\nu}^{T})
=k[a_{ij},b_{ik}]/I(M)$$
where $I(M)$ is the ideal generated by the degree conditions: $$ deg(\prod_{i\in I_{\gamma}^{+}} 
(a_i)^{\gamma_i^+}\prod_{i\in I_{\gamma}^{-}} 
(b_i)^{\gamma_{i}^{-}})\leq (\gamma,\nu)+D_{-\gamma}(M)$$ for each $\gamma\in \Gamma$ and the conditions $a_{i}b_{i}=1$ for each i in I.

The conjecture 
$\cO (X(M)_{\nu}^{T}) \cong H^*(Gr_{e}^{\Pi}(M))$,
where $e_i=(\nu, \varpi_i)$,
is now equivalent to

$$k[a_{ij},b_{ik}]/I(M)
\cong H^*(Gr_{e}^{\Pi}(M)).$$ 

The quiver Grassmannian $Gr_{e}^{\Pi}(M))$ is a subvareity of 
$\prod_{i\in I} Gr_{e_{i}}(M_i)$ and we have on each $Gr_{e_{i}}(M_i)$ the tautological subbundle $S_i$ and quotient bundle $Q_i$. We pull back $S_i$ and $Q_i$ to $\prod_{i\in I} Gr_{e_{i}}(M_i)$ and denote their restrictions  on $Gr_{e}^{\Pi}(M))$ still by $S_i$ and $Q_i$ by abusing notion.
For a rank n bundle E, denote the Chern class by $c(E)$  and the $i^th$ Chern class $c_i(E)$, where $c(E)=1+c_1(E)+\cdots+c_n(E) $.
We want to define the map $$\Psi: \cO (X(M)_{\nu}^{T}) \xrightarrow[]{} H^*(Gr_{e}^{\Pi}(M)),
\text{ where } e_i=(\nu, \varpi_i),$$ by mapping the generators $a_{ij}$ to $ c_{j}(S_i)$ and 
$b_{ij}$ to $c_{j}(Q_i)$.

\begin{theorem}
The map $\Psi$ described above is well defined.

\end{theorem}
\subsection{Two lemmas}
For the proof, we need two lemmas. Lemma 1 is the special case of theorem 4 when $Q$ is the quiver $1\xrightarrow[]{}2$
and $M$ is a $kQ$-module.
\begin{lemma}
 Let $Q$ be the quiver $1\xrightarrow[]{}2$
and $M$ be $\CC^{d_1}\xrightarrow[]{\phi} \CC^{d_2} $.On $X=Gr^{\Pi}_{e}(M)$, we have $c_i(S_2\oplus Q_1)=0 $ when $i> e_2-e_1+\text{dim}(ker\phi)$.

\end{lemma}
Let $\phi_{ij}: M_i\xrightarrow[]{} M_j$ be the composition of $\phi_a$ where a travels over the unique no going-back path which links i and j.  
Let $M_\gamma= \oplus_{i\in I_{\gamma}^{-}}M_i^{\gamma^{-}} \xrightarrow[]{\phi_{\gamma}=\oplus\phi_{ij}}\oplus_{i\in I_\gamma^{+}}M_i^{\gamma^{+}}$ be the module over $k(1\xrightarrow[]{}2)$.

\begin{lemma}
For a $\Pi$-module $M$ and any chamber weight $\gamma$, we have
 $$dim(ker\phi_{\gamma})= D_{-\gamma}(M).$$
\end{lemma} 
Lemma 2 is a property of $D_{\gamma}$ and will be proved in the appendix.
\subsection{Proof of theorem 4 from lemmas in $\mathsection$ 4.2}
\begin{proof}[Proof of theorem 4]
We prove the theorem can be reduced to lemma 1.

For each $\gamma \in \Gamma$,we have to prove the degree inequilities carry over to Chern classes:
$$\Psi(\prod_{i\in I_{\gamma}^{+}} t_i^{\gamma_{i}^+}\prod_{i\in I_{\gamma}^{-}} s_i^{\gamma_{i}^{-}})=\prod_{i\in I_{\gamma}^{+}} c(S_i)^{\gamma_{i}^+}\prod_{i\in I_{\gamma}^{-}} c(Q_i)^{\gamma_{i}^{-}}\leq D_{w_0\gamma}(M)+(\nu,\gamma). $$

Define a map $\Phi$ from $Gr^{\Pi}(M)$ to $Gr^{k(1\xrightarrow[]{}2)}(M_\gamma)$: for $N\in Gr^{\Pi}(M)$,
$\Phi(N)=\oplus_{i\in I_{\gamma}^{-}}N_i \xrightarrow[]{\phi_{\gamma}}\oplus_{i\in I_\gamma^{+}}N_i$. We have 
$$\Phi^{*}(c(S_2)c(Q_1))=c(\Phi^{*}(S_2))c(\Phi^{*}(Q_1))= c(\oplus_{i\in I_{\gamma}^{+}}S_i^{\gamma_{i}^{+}})c(\oplus_{i\in I_{\gamma}^{-}}Q_i^{\gamma_{i}^{-}})$$

$$=\prod_{i\in I_{\gamma}^{+}} c(S_i)^{\gamma_{i}^{+}}\prod_{i\in I_{\gamma}^{-}} c(Q_i)^{\gamma_{i}^{-}}.$$
Apply lemma1 to $M_{\gamma}$ we have $$deg(c(Q_1)c(S_2)) \leq dimker(\phi_{\gamma})+\sum_{i\in I_{\gamma}^{+}}\gamma_i e_i-\sum_{i\in I_{\gamma}^{-}}\gamma_i e_i$$
$$=dimker(\phi_{\gamma})+\sum_{i\in I}\gamma_i e_i=dimker(\phi_{\gamma})+(\gamma,\nu).$$
Then the theorem follows by lemma 2.
\end{proof}

Chern class vanishes in certain degree when the bundle contains a trivial bundle of certain degree but the desired trivial bundle in $Q_1\oplus S_2$ does not exist.
The idea is to pass to T-equivariant cohomology. Over $X^T$ which is just a union of isolated points we will decompose $Q_1\oplus S_2$ into the sum of the other two bundles $E_1$ and $E_2$ pointwisely, where $E_{1}$ will play the role of trivial bundle. Although there is no bundle over X whose restriction is $E_2$, there exist T-equivariant cohomology class in $H_{T}^{*}(X)$ whose restriction on $X^T$ is the T-equivariant Chern class of $E_2$. 
\subsection{Recollection of GKM theory}
We first recall some facts in  T-equivariant cohomoloy theory.

We follow the paper \cite{tymoczko2005introduction}.
Denote a n-dimensional torus by $T$, topologically $T$ is homotopic to$ (S^1)^n$. Take $ET$ to be a contractible space with a free
$T$-action.
Define $BT$ to be the quotient $ET/T$.
The diagonal action of $T$ on $X \times ET$ is free, since the action
on $ET$ is free.  Define $X \times_T ET$ to be the quotient 
$(X \times ET) / T$.  
We define the equivariant cohomology of $X$ to be
\[H^*_T(X) = H^*(X \times_T ET).\]

When $X$ is
a point and $T=G_m$,
$$  
H^*_T(X) = H^*(\textup{pt} \times_T ET) = 
H^*(ET/T)=H^*(BT)=
H^*(\mathbb{CP}^{\infty})\cong k[t].$$ When $T=(S^1)^n$,
\begin{align}
H^*(pt)=k[t_1,\cdots,t_n]\cong S(\mathfrak{t}^*).
\end{align}
So we can identify any class in $H^*(pt)$ as a function on the lie algebra $\mathfrak{t}$ of $T$.
The map $X \xrightarrow[]{}pt$ allows us to pull back each class in
$H^*_T(\textup{pt})$ to $H^*_T(X)$, so $H^*_T(X)$ is a
module over $H^*_T(\textup{pt})$.  

Fix a projective variety $X$ with an action of $T$. We say that $X$ 
is equivariantly formal with respect to this $T$-action if
$E^2=E^\infty$ in the spectral sequence associated to the fibration
$X \times_T ET \longrightarrow BT$.

When $X$ is equivariantly formal with respect to $T$, 
the ordinary cohomology of $X$ can be reconstructed from its
equivariant cohomology. Fix an inclusion map  
$X\xrightarrow[]{}X \times_T ET$, we have
the pull back map of cohomologies: $H^*(X \times_T ET\xrightarrow[]{i} H^*(X)$.
The kernel of $i$ is $\sum_{s=1}^{n} t_s\cdot H^*_T(X)$, where $t_s$ is the generator of $H^*_T(pt)$ (see (4)) and we view it as an element in $H^*_T(X)$ by pulling back the map $X\xrightarrow[]{}pt$. Also $i$ is surjective so
$H^*(X) = H^*_T(X)/ ker(i)$.

If in addition $X$ has finitely many fixed points and finitely many one-dimensional orbits, 
Goresky, Kottwitz, and MacPherson show that
the combinatorial data encoded in the graph of fixed points
and one-dimensional orbits of $T$ in $X$
implies a particular algebraic characterization of $H^*_T(X)$.

 \begin{theorem}[GKM, see \cite{tymoczko2005introduction}, \cite{goresky1997equivariant}]
Let $X$ be an algebraic variety with a $T$-action with respect to which
$X$ is equivariantly formal, and which has 
finitely many fixed points and finitely
many one-dimensional orbits.  Denote the one-dimensional orbits $O_1$, $\ldots$, $O_m$.  For each
$i$, denote the the $T$-fixed points of $O_i$ by $N_i$ and $S_i$ and denote the 
stabilizer of a point in $O_i$ by $T_i$. Then the map
$H^{*}_{T}(X)\xrightarrow[]{l}H^{*}_{T}(X^T)=\oplus_{p_{i}\in X^{T}}H^*_{T}(p_{i})$  is injective
and its image is 
$$\left\{f=(f_{p_1}, \ldots, f_{p_m}) 
\in \bigoplus_{\textup{fixed pts}} S(t^*):  f_{N_i}|_{\mathfrak{t}_i} = 
f_{S_i}|_{\mathfrak{t}_i} \textup{ for each }i=1,\ldots,m \right\}.$$
Here $\mathfrak{t}_i$ is the lie algebra of $T_i$.
\end{theorem}
\subsection{Affine paving of $Gr^{\Pi}_e(M)$ when $M$ is a representation of $Q$ of type A}

\begin{defi}[\cite{affine} 2.2]
We say a space $X$ is paved by affines if $X$ has an order partition into disjoint $X_1,X_2,\cdots$ such that each finite union $\bigcup_{i=1}^j X_i$ is closed in X and each $X_i$ is an affine space. 
\end{defi}
A space with an affine paving has odd cohomology vanishing.
\begin{prop}[\cite{affine}, 2.3]
Let $X=\bigcup X_i$ be a paving by a finite number of affines with
each $X_i$ homeomorphic to $\CC^{d_i}$.  The cohomology groups
of $X$ are given by $H^{2k}(X) = \bigoplus_{\{i\in I \ |\ d_i = k\}} \mathbb{Z}$.
\end{prop}
%%Since $\Lambda(d)$ and $D_{\gamma}$ do not depend on the direction of the edges in E, we fix the quiver Q with all edges in E pointing to the right in in type A.
The main observation is the following lemma.
\begin{lemma}
Let $M$ be a representation of Q, where Q is of type A with all edges in E pointing to the right. Then the quiver Grassmannian $Gr^{\Pi}_e(M)$
is paved by affines for any dimension vector e.
\end{lemma}

We need a sublemma first.

\begin{sublemma}

Suppose X is paved by $X_i$'s. Let $Y\subset X $ be a subspace. if for each i, $Y_i=X_i \bigcap Y$ is $\emptyset$ or affine then $Y=\bigcup Y_i$ is an affine paving.\\
\end{sublemma}
\begin{proof}
 $\bigcup_{i\leq j} Y_i=\bigcup_{i\leq j} (X_i \bigcap Y)=(\bigcup_{i\leq j} X_i) \bigcap Y$ is closed in $Y$ since $\bigcup_{i\leq j} X_i$ is closed in $X$.
\end{proof}
\begin{proof}[Proof of lemma 3]

Let $V=\oplus M_i$ be the underlying vector space and $\phi=\oplus_{a\in H} \phi_a$ be the nilpotent operator on $V$.
We adopt the notations in \cite{shimomura}.
Let $n=dimV$ and $d=\sum e_i$.
Let $C^{\phi}_\alpha =\{(v_1,...v_d)\in C_{\alpha}, \phi v_i=v_j $ if $\alpha$ contains 
$\ytableaushort[\alpha_]{ij}$
$,(1\leq i <j<d)\}$.
We know that $Gr_{d}(k^n)^{\phi}=\bigsqcup C_{\alpha}^{\phi}$.
We want to show $S^{\phi}_\alpha \bigcap Gr^{\Pi}_{e}(M)$ is affine.
Take $x\in S^{\phi}_\alpha \bigcap Gr^{\Pi}_{e}(M)$, from 1.10 in \cite{shimomura}, $x=v_1\wedge...\wedge v_d$ where 
$x= \wedge(\phi^{h}w_i : i's $ are the initial numbers of $\alpha$ and $\phi^{h}w_i\ne 0).$\\
We now show that
for $x=v_1\wedge\cdots \wedge v_d \in Gr^{\Pi}_{e}(M)$, where $v_i=e_i+ \sum_{j\neq i} x_i(j)e_j $, if $e_i\in M_t$, we have $v_i \in M_t$. Conversely, if for each i, there exists t such that $v_i \in M_t$, $x\in Gr^{\Pi}_{e}(M)$.Denote this t determined uniquely by i as t(i).\\
Since $span(v_1,...,v_d)$ is a direct sum of some $N_i\subset M_i$, we have $Pr_t(v_i) \in span(v_1,...,v_d)$,
where $Pr_t$ is the projection from V to $M_t$, 
and so $Pr_t(v_i)=\sum a_{p}v_{p}$. Comparing the coefficient of $e_p$, by the definition of $C_{\alpha}$, we have $a_p=0$ for $p\neq i$.
So we have $Pr_t(v_i)=v_i$, which implies $v_i\in M_t$. \\
Note that $\{v_1,\cdots,v_d\}$ is determined by $w_i$ where i is an initial number(and vice versa).\\
We have $w_i\in M_{t(i)}$.
Denote $l(i)$ be the number on the left of i in the d-tableaus. If i is the leftmost, set $l(i)$ to be $\emptyset$, and set $e_{\emptyset}=0$. Write $w_i= e_i +\sum x_{ij}e_j$, where $e_j\in M_{t_i}$,
we have $\phi^{r}(w_i)= e_{l^{r}(i)} +\sum x_{ij}e_{l^{r}(j)}$. Since 
M is kQ-module, we have $l^{r}(i)=l^{r}(j)$ hence $v_{i_{r}}\in M_{t(i_{r})}$ and $x\in Gr^{\Pi}_e(M)$. So we have 
$S^{\phi}_\alpha \bigcap Gr^{\Pi}_{e}(M)$ is affine.
Apply lemma 3, we are done.

\end{proof}
\subsection{Proof of lemma 1}
\begin{proof}
For $M$ given by $\CC^{d_{1}}\xrightarrow[]{\phi} \CC^{d_{2}}$ and a choice of $e=(e_1, e_2)$,
denote $X=Gr^{\Pi}_e(M)$.
First, we define a torus action on $X$.
Let $I=ker\phi $. Choose a basis $e_1,e_2,\cdots ,e_s $ of $I$ and extend it to a basis $e_1,\cdots ,e_s,e_{s+1},\cdots,e_t$ of $M_1$. Let $J$ be span$\{e_{s+1},\cdots,e_t)\}$ so the image of $J$ is span $\{f_{s+1},\cdots,f_t\}$. We extend the basis $\{f_i=\phi(e_i)\}$ of the image of $J$ to a basis ($f_{s+1},\cdots,f_t,f_{t+1},\cdots,f_r)$ of $M_2$.  Let K=span$\{f_{t+1},...f_r\}$.
we have $M_1=I\oplus J$ and $M_2=\phi(J)\oplus K$.\\

Let $\cI=\{1,\cdots,s\}$, $\cJ=\{s+1,\cdots,t\}$ and $ \cL=\{t+1,\cdots,r\}$.
Let tori $T_{I}=G_m^{\cI},T_{J}=G_m^{\cJ} ,T_{L}=G_m^{\cL}$
act on $I,J\cong \phi(J)$, $K$
by multiplication compotentwisely (For instance, $T_{I}$ acts on I by $(t_1,\cdots,t_s)\sum a_{i}e_i =\sum a_{i}t_{i}e_i$ and on $J,K$ trivially). Hence they act on $M_1=I\oplus J$ and $M_2=\phi(J) \oplus K$.
This induces an action of $T=T_{I}\times T_{J}\times T_{K}$ on $Gr^{\Pi}_e(M)$.
By lemma 3, $Gr^{\Pi}_e(M)$ is paved by affines so by proposition 3 it has odd cohomology vanishing therefore the spectral sequence associated to the fibration
$X \times_T ET \longrightarrow BT$ converges at $E^2$ and X is equivariantly formal.

Denote by $f$ the forgetful map $H^*_{T}(X)\xrightarrow[]{f} H^*(X)$. From $\mathsection 4.4 $ we have $
ker(f)=\sum_1^{dimT} t_s H^*_{T}(X)$.
Since $c^i(S_2\oplus Q_1)=f(c^{i}_{T}(S_2\oplus Q_1))$,
it suffices to prove $c^{i}_{T}(S_2\oplus Q_1))\in ker(f)$
when $i>e_2-e_1+dimI$.

To use GKM theorem, we need to know the one dimensional orbits and T-fixed points of X.

First, we see what $X^T$ is.
For a point $p=(V_1, V_2)$ in $X$, in order to be fixed by $T$, $V_1$ and $V_2$ need to be spanned by some of  basis vectors $e_i$ and $f_i$.
For a subset $S$ of $\cI$ (resp. $\cJ$), we denote by $e_{S}$ (resp. $f_{S}$) the span $\{e_i|i\in S\}$ (resp. span$\{f_i|i\in S\}$). The T-fixed points in $X$ consist of all $V=(V_1, V_2)$, such that $V_1=e_{A\bigcup B}, V_2=f_{C\bigcup D}$, for some
$A\subset I,B\subset C\subset J$ and $D\subset K$.

For any point $p=(V_1, V_2)$ in $X^T$, let $V_1=e_{A\bigcup B}, V_2=f_{C\bigcup D}$.
Over $p$, $Q_1=(I\oplus J)/ e_{A\bigcup B}$  is isomorphic to $e_{(\cI \setminus A)\oplus (\cJ  \setminus B) }$
(The restriction of a $T$-equivariant bundle to a T-fixed point is just a $T$-module).
%%and in this case these two vector spaces are isomorphic T-equivariantly(commutes with $T$ action).
So over $X^T$, we can decompose $S_2\oplus Q_1$ as follows:
$$S_2\oplus Q_1 \cong e_{(\cI \setminus A)\oplus (\cJ  \setminus B)}\oplus f_{(C\bigcup D)}=(e_{(\cI \setminus A)\oplus (C  \setminus B)}\oplus f_{D}) \oplus (e_{\cJ \setminus C}\oplus f_{C}).$$
Denote the bundle over $X^{T}$ whose fiber over each point p is $e_{(\cI \setminus A)\oplus (C  \setminus B)}\oplus f_{D})$ by $E_1$ and the bundle over $X^{T}$ whose fiber over p is $e_{(\cJ \setminus C)}\oplus f_{C}$ by $E_2$.

We now use localization. Denote by $l$ the map 
$H^{*}_{T}(X)\xrightarrow[]{l}H^{*}_{T}(X^T)=\oplus_{p\in X^T} H^*(p)$. From GKM theory $l$ is injective, so the condition $c^{i}_{T}(S_2\oplus Q_1))\in ker(f)$ is equivalent to 
$l(c^{i}_{T}(S_2\oplus Q_1))\in l(ker(f))$.
We have
\begin{align}
l(ker(f))=l(\sum_{s=1}^{dimT} t_s H^*_{T}(X))=\sum_{s=1}^{dimT} \underbrace{(t_s,\cdots, t_s)}_\text{the number of T-fixed points in X} l(H^*_{T}(X)).
\end{align}
By functorality of Chern class, 
$l(c^{i}_{T}(S_2\oplus Q_1))=c^{i}_{T}(S_2|_{X^T}\oplus Q_1|_{X^T}))$.\\
We compute the $T$-equivariant Chern class over $X^{T}$. For\footnote{We always denote S and Q but indicate over which space we are considering.} each p,
$$c_{T}^p (S_2\oplus Q_1)=c_{T}^p (E_2\oplus E_1)=\sum_{i} c_{T}^{p-i} (E_1) c_{T}^i (E_2)=\sum_{i\geq 1}c_{T}^{p-i} (E_1) c_{T}^i (E_2) .$$ The last equality holds since $c_{T}^{p} (E_2)=0$ when $p> dimE_2=dimI+e_2-e_1$.
Now to show $c_{T}^p (S_2\oplus Q_1)\in l(ker(f))$, It suffices to show that $c_{T}^{p-i} (E_1) c_{T}^i (E_2)\in l(ker(f))$, for any $i$.
The action of T on $E_2$ is actually the same on each T-fixed point. And at each point, $c_{T}^i (E_2)$ is the $i^{th}$ elementary symmetric polynomial of $t_s, 1\leq s\leq dimT$. So by (5), it suffices to show that $c_{T}^{i} (E_1)\in l(H_{T}^{*}(X))$.

Now we will see what 1-dimensional orbits are. Take an orbit 
$O_i$, in order to be 1 dimensional its closure must contain two fixed points. Let $\overline{O_i}=O_i\bigcup \{N_i\}\bigcup \{S_i\}$, where
$N_i=(e_{A\bigcup B}, f_{C\bigcup D})$ and $S_i=(e_{A^{\prime}\bigcup B^{\prime}}, f_{C^{\prime}\bigcup D^{\prime}})$ are the fixed points. 
$O_i$ is one dimensional whenever either $A\bigcup B$ and $A^{\prime}\bigcup B^{\prime}$ differ by one element with $C\bigcup D=C^{\prime}\bigcup D^{\prime}$ or $C\bigcup D$ and $C^{\prime}\bigcup D^{\prime}$ differ by one element with $A\bigcup B=A^{\prime}\bigcup B^{\prime}$ . In the first case, we have some $s\in A\bigcup B$ and $s^{\prime} \in A^{\prime}\bigcup B^{\prime}$, such that $A\bigcup B \setminus s =A^{\prime}\bigcup B^{\prime} \setminus s^{\prime}$.

Notice that the annihilator for the lie algebra $\mathfrak{t}_i$ in $S(t^*)$ is generated by $t_s-t_{s^\prime}$, so by theorem 5, the condition along $O_i$ for an element $h\in H^*_{T}(X^T)$ to be in $im(l)$ is $$(t_s-t_{s^{\prime}})\ |\  (h_{N_i}-h_{S_i}).$$
But we have 
$$c_{T}(E_1)|_{N_i}-c_{T}(E_1)|_{S_i}=\\
(1+t_{s{\prime}})\prod_{i\in \cI\bigcup C \setminus (A\bigcup B)\setminus \{s^{\prime}\}} (1+t_i)-
(1+t_s)\prod_{i\in \cI\bigcup C \setminus (A^{\prime}\bigcup B^{\prime})\setminus \{s\}} (1+t_i).$$
Note that $\cI\bigcup C \setminus (A\bigcup B)\setminus \{s^{\prime}\} = \cI \bigcup C \setminus (A^{\prime}\bigcup B^{\prime})\setminus \{s\}$, 
so $t_s-t_{s^{\prime}} $ divides $c_{T}(E_1)|_{N_i}-c_{T}(E_1)|_{S_i}$.
We conclude that $c_{T}^{i} (E_1)\in l(H_{T}^{*}(X))$.\\
The other case is similar.

\end{proof}
\section{Proof of isomorphism when $M$ is a representation of $Q$ of type A}

We first prove that $\Psi$ is surjective.

\begin{lemma}
(a) Denote $Y=\prod Gr_{e_i} (k^{d_i})$ and $X=Gr^{\Pi}_{e}(M)$. Then $Y\setminus X $ is paved by affines.\\
(b) $\Psi$ is surjective.
\end{lemma}

\begin{proof}

(a).
Let a be the number where the Young diagram of $\phi_{Y}$ has $a^{th}$ row as the first row from the bottom that does not have one block. For example, in the left diagram, a=4.
$$\ydiagram{4,3,2,2,1,1,1}\xrightarrow[]{} \ydiagram{4,3,2,1,1,1,1,1}$$
Define $\phi^{\prime}$ be the operator of $V$ that corresponds to the diagram by moving the left most block A of the $a^{th}$ row to the bottom in the diagram of $\phi_{Y}$.
Let $M^{\prime}$ be the corresponding module and $X^{\prime}$ be $Gr_e{M^{\prime}}$.\\
We claim that $X^{\prime}\setminus X$ is paved by affines.\\
By lemma 4, we have $X=\bigsqcup_{\alpha \in I } C_{\alpha}$, where $I^{\prime}$ is the set of all semi-standard young tableau in $\lambda$. Also we have $Y^{\prime}=\bigsqcup_{\alpha \in I^{\prime} } C^{\prime}_{\alpha}$, where $I$ is the set of all semi-standard young tableau in $\lambda^{\prime}$.
If $\alpha$ contains block A, $\alpha$ is s.s in $\lambda$ implies $\alpha^{\prime}$ is s.s in $\lambda^{\prime}$.
If $\alpha$ does not contain block A , $\alpha$ also does not contain any block in that row, so $\alpha$ is still s.s in $\lambda^{\prime}$. So $I\subset I^{\prime}$.
\\
For $\alpha$ that contains block A, there are two types.
Let E be the set of $\alpha$ that contains block A and some other block in the row of A.
Let F be the set of $\alpha$ that contains block A but no other block in the row of  A.
Let G be the set of $\alpha $ that does not contain block A.
So we have $I=E\bigsqcup F\bigsqcup G=F\bigsqcup (E\bigsqcup G)$.\\
Take $\alpha \in F$, in $\lambda$,
the block A in $\alpha$ is not initial so 
the vector indexed by A is determined by the initial vector.
In $\lambda^{\prime}$, A is the last block so the vector indexed by A is the basis vector indexed by block A. In both case the vector indexed by A has been determined, so $C_{\alpha}=C_{\alpha}^{\prime}$ when $\alpha \in F.$

For $\alpha \in E\bigcup G$, let $s(\alpha) $ be the tableau of the same relative position in $\lambda^{\prime}$ as $\alpha$ in $\lambda$. Then $C_{\alpha}=C^{\prime}_{s(\alpha)}$. Since s is a bijection between $E\bigsqcup G$ and $E^{\prime}\bigsqcup G^{\prime}$ we have $\bigsqcup_{\alpha \in E\bigsqcup G} C_{\alpha}=\bigsqcup_{\alpha \in E^{\prime}\bigsqcup G^{\prime}} C^{\prime}_{\alpha}$ .\\
Then we have $X^{\prime}\setminus X=\bigsqcup_{\alpha \in I^{\prime}\setminus I} C^{\prime}_{\alpha}$ is paved by affines.
We can do this procedure step by step until $X^{\prime}$ becomes Y, so we are done.

(b). By lemma 1, X is paved. With part(a) , we have the homology map from X to Y is injective hence $\Psi$ is surjective (see 2.2 in\cite{affine}).
\end{proof}

We want to prove the two sides of $\Psi$ have the same dimension as k-vector spaces and actually we will prove it for a more general setting.
\begin{defi}
For a $\Pi$-mod $M$
Let $V=\oplus M_i$ be the underlying vector space and $\phi=\oplus_{a\in H} \phi_a$ be the nilpotent operator on $V$.
We say $M$ is $I$-compatible if there is a Jordan basis $\{v_{ij}\}$ of V such that each $v_{ij}$ is contained in some $M_r$.

\end{defi}
\begin{defi}
For a $\Pi$-module $M$, if the Young diagram of the associated operator $\phi$ has one raw, we call $M$ one direction module.
\end{defi}
\begin{prop}
If $M$ is I-compatible, it is a direct sum of one direction module with multiplicities.
\end{prop}

\begin{lemma}
If M is I-compatible, the dimension of two sides of $\Psi$ have the inequality:
dim $k[a_{ij},b_{ik}]/I(M,e))\leq \chi (Gr^{\Pi}_e(M))$, where $\chi$ is the Euler character. 
We denote the ring on the left by $R(M,e)$.
\end{lemma}
First we state a lemma due to Caldero and Chapoton.
\begin{lemma}[see prop 3.6 in \cite{caldero2004cluster}]
For $\Pi$-module $M,N$, we have
$$\chi(Gr_{g}(M\oplus N)=\sum_{d+e=g} \chi(Gr_{d}(M))\chi(Gr_{e}(N)).$$
\end{lemma}
The following two lemmas are proved after the proof of lemma 5.
\begin{lemma}
$R(M,e)/ <b_{1(d_1-e_1)}> \cong R(M^{\prime},e)$.

\end{lemma}
\begin{lemma}
$b_{1(d_1-e_1)}R(M,e)$ is a module over $R(M^{\prime\prime},e-\sum_{i \in I} \alpha_i)$.

\end{lemma}
\begin{proof}[Proof of lemma 5]\footnote{The rest of this chapter is not well-written and  should be revised.}

We index the basis vector according to the Young diagram of $\phi$ as before but slightly different: $e_{ij}$ corresponds to the block of $i^{th}$ row (from up to down) and $j^{th}$ column (from right to left, which is the difference from before and this will cause the problem that two blocks in the same column but different row have different j but we will fix an i sooner so will not be of trouble). so  $\phi(e_{ij})=e_{i(j-1)}$.\\
The basis vector in $M_1$ appears in the first column or in the last column. If it lies in the last, we can take the dual to make it in the first. So,
there exist i such that $e_{i1}\in M_1$. \\Let $\lambda^{\prime}$ be the young diagram removing the block of $e_{i1}$ from the original one and $M^{\prime}$ be the corresponding module.
Let $\lambda^{\prime \prime}$ be the young diagram removing $i^{th}$ row and $M^{\prime \prime}$ be the corresponding module.
Apply lemma 6, we have
$$\chi(Gr^{\Pi}_e(M))=\chi(Gr_{e}(M^{\prime}))+\chi(Gr_{e-\sum \alpha_i}(M^{\prime \prime})).$$
We count the dim of $R(M,e)$ by dividing it into two parts.\\$$
dimR(M,e)=dim b_{1(d_1-e_1)} R(M,e)+ dim R(M,e)/ b_{1(d_1-e_1)} R(M,e).$$

By lemma 7 and 8, since $b_{1(d_1-e_1)}R(M,e)$ is acyclic, \\
$$dim b_{1(d_1-e_1)}R(M,e) \leq dimR(M^{\prime\prime},e-\sum_{i \in I} \alpha_i).$$So 
$dim R(M,e)=dim b_{1(d_1-e_1)} R(M,e))+ dim R(M,e)/ b_{1(d_1-e_1)} R(M,e)\\
\leq dim R(M^{\prime\prime},e-\sum_{i \in I} \alpha_i)
+dim  R(M^{\prime},e)=\chi (Gr_{e}(M^{\prime})+ \chi (Gr_{(e-\sum_{i \in I} \alpha_i) }  (M^{\prime \prime}))=
\chi (Gr^{\Pi}_{e}(M))$.
\end{proof}
Now we prove lemma 7 and 8.
\begin{proof}[Proof of lemma 7]

Recall $I(M)=\text{deg}\prod_{i\in I{+}} 
(t_i)^{\gamma_i}\prod_{i\in I{-}} 
(s_i)^{\gamma_i}\leq (\gamma,\nu)+D_{-\gamma}(M) $
We denote $v(M,\gamma,e)=(\gamma,\nu)+D_{w_{0}\gamma}(M)$.
The difference between $I(M)$ and $I(M)^{\prime}$ only occurs when $\gamma= -\varpi_1$. In this case $v(M^{\prime},-\varpi_1,e)=v(M,-\varpi_1,e)-1$. The degree of $s_1$ goes down by by 1, meaning we have one more vanishing condition which is $b_{1(d_1-e_1)}=0$.
\end{proof}
\begin{proof}[Proof of lemma 8]

In order to define a module structure on $b_{1(d_1-e_1)}R(M,e)$, we lift the element in $R(M^{\prime\prime},e-\sum_{i \in I} \alpha_i)$ to $R(M,e)$ (since the former is a quotient of the latter) and let it act on $b_{1(d_1-e_1)}R(M,e)$ by multiplication.
We denote J to be the degree  $v(M,\gamma,e)$ part of $<\prod_{i\in I{+}} 
(t_i)^{\gamma_i}\prod_{i\in I{-}} 
(s_i)^{\gamma_i},\gamma\in \Gamma>$.\\

We need to check it is independent of the choice of the lift: $$b_{1(d_1-e_1)}J \subset I(M,e), \textup{ where }
I(M,e)=J\oplus I(M^{\prime\prime},e-\sum_{i \in I} \alpha_i).$$

Denote the module corresponding to $i^{th}$ row P.
 We have $M=M^{\prime\prime}\oplus P$. Then 
$v(M,\gamma,e)-v(M^{\prime\prime},\gamma,\sum_{i \in I} \alpha_i)=v(P,\gamma,\sum_{i \in I} \alpha_i)$.\\
We claim that $v(P,\gamma,\sum_{i \in I} \alpha_i)=0 \ or \ 1$
and is 0 when $\gamma_{1}=1$.\\
%pf:( only checked for type A,D(?))
This is a direct calculation.
%(will need a more conceptual proof).\\

When $\gamma_{1}=-1$, $t_{1j}$ appears in each summand  $\prod_{i\in I{+}} 
(t_i)^{\gamma_i}\prod_{i\in I{-}} 
(s_i)^{\gamma_i},\gamma\in \Gamma$. Let one of the summand be 
$t_{1j}k$.\\
$s_{1(d_1-e_1)} t_{1j} k=-\sum_{p+q=d_1-e_1+j} s_{1p}t_{1q}k
=-s_{1(d_1-e_1+j-q)}\sum_{q>j} t_{1q}k$.
We have $\sum_{q>j} t_{1q}k$ is in $I(M,e)$ since this is of degree larger than $v(M,\gamma,e)$.\\
When $\gamma_1 =0$,\\
we claim that when 
$v(P,\gamma,\sum_{i \in I} \alpha_i)$ is 1 , we have $\gamma-\varpi_1 \in \Gamma$. \\
Then we want to show 
$s_{1(d_1-e_1)}\prod_{i\in I{+}} 
(t_i)^{\gamma_i}\prod_{i\in I{-}} 
(s_i)^{\gamma_i}$ is in $I(M,e)$.
Let k is a summand of degree $v(M,\gamma,e)$
part of $I(M,e)$.
We want to show $s_{1(d_1-e_1)}k$ is of degree $v(M,\gamma-\varpi_1,e)+1$.
So we need $v(M,\gamma,e)+d_1-e_1\geq v(M,\gamma-\varpi_1,e)+1$. By the
dimension description, the image of $\Phi_{\gamma-\varpi_1}$ is 
at least 1 dimensional bigger than the image of $\Phi_{\gamma}$ since $\phi_{1m}(e_{i1})= e_{im}$ is in the image but for $\gamma-\varpi_1$ (since $\phi_{1m}$ is not a summand of $\phi$) the projection of $img(\phi)$ on $V_m$ is zero, where m is the smallest number in $\Gamma_+$.
\end{proof}
\begin{theorem}
$\Psi$ is an isomorphism when M is a kQ-module.
\end{theorem}
\begin{proof}
by lemma 4, $\Psi$ is surjective so dim $k[a_{ij},b_{ik}]/I(M,e))\geq \chi (Gr^{\Pi}_e(M))$ and by lemma 5 this is an equality so the theorem follows.
\end{proof}

\section{A consequence of this conjecture}
%%Global loop Grassmannian  

In section 3, we defined $\cG(G,Y)$ as moduli of maps of between pairs form $(d,d^*)$ to $(G/Y,pt)$.
This is actually a local version of (fiber at a closed point c) the global loop Grassmannian with a condition Y to a curve $C$, $\cG^{C}(G,Y)$.
To a curve C, define $\cG^{C}(G,Y)$ over the ran space $\cR_C$ with the fiber at $E\in \cR_C$:
$$\cG^{C}(G,Y)_{E}=^{def} map[(C,C-E), (G/Y, pt)].$$
Denote the map from $\cG^{C}(G,Y)$ to $\cR_C$  remembering the singularities by $\pi$.

One can ask if $\pi$ is (ind) flat for any $G$ and $(Y,pt)$.
The case we are concerned is when $G^{\prime}=G\times \prod_{w} T_w$ and $Y=\prod_{w}(G/N^{w})^{aff}.$
Let $c\in C, \underline{\lambda_w}, \underline{\mu_w}\in X_*(T)^{W}.$
In particular, we restrict $\cG^{C}(G^{\prime},Y)$ to $C\times c$ and denote the image under projection from $\cG(G^{\prime}$ to $\cG(G)$ by $X$.
We have $X$ is a closed subscheme of $Gr_{G, X\times c}$.
Explicitly, an $R$-point of $Gr_{G, X\times c}$ consists of the following data
\begin{itemize}
\item $x: specR\xrightarrow[]{} C$. Let  $\Gamma_x$
be the graph of x. Let $\Gamma_c$ be the graph of the constant map taking value c.
\item
$\beta$ a G-bundle on spec$R\times C$.
\item
A trivialization $\eta: \beta_0 \xrightarrow[]{\eta} \beta $ defined on $specR\times C- (\Gamma_x\bigcup \Gamma_c)$.

\end{itemize}

An $R$-point of $X$ over $C\times c$ consists of an $R$-point of $Gr_{G, X\times c}$ subject to the condition:
For every $i\in I$,
the composition
$$\eta_i : \beta_0 \times^{G} V(\varpi_i) \xrightarrow[]{}  \beta \times^{G} V(\varpi_i) \xrightarrow[]{} \beta \times^{G} V(\varpi_i) \otimes \cO(\langle \gamma, \lambda_w \rangle \cdot \Gamma_x + \langle \gamma, \mu_w \rangle \cdot \Gamma_c) .$$
is regular on all of $specR\times C$.

We can show the fiber over a closed point other than $c$ is $\bigcap \overline{S^{w}_{\lambda_w}}\times \bigcap \overline{S^{w}_{\mu_w}}$ and the fiber over $c$ is $\bigcap \overline{S^{w}_{\lambda_w+\mu_w}}$.

\begin{cor}[Given the conjecture]
The T-fixed point subscheme of this family is flat.

\end{cor}

\begin{proof}
$dim \cO((\overline{ \bigcap S^{w}_{\lambda_w+\mu_w}}^{T})_{\nu})=dim H^*(Gr^{\Pi}_e (M))=\footnote{By lemma 6, and given the conjecture, Euler character is the same as total cohomology.}\sum_{e_1+e_2=e} dim H^*(Gr^{\Pi}_{e_1} (M_1) dim H^*(Gr^{\Pi}_{e_2} (M_2)= \sum_{\nu_1 +\nu_2=\nu } \cO(\overline {\bigcap S^{w}_{\lambda_w}}^{T}_{\nu_1})  \cdot \cO(\overline{ \bigcap S^{w}_{\mu_w}}^{T}_{\nu_2})
=dim \cO( 
\bigsqcup_{\nu_1 +\nu_2=\nu } \overline{ \bigcap S^{w}_{\lambda_w}}^{T}_{\nu_1}\times \overline{ \bigcap S^{w}_{\mu_w}}^{T}_{\nu_2}=
dim (\overline{ \bigcup S^{w}_{\lambda_w}}^{T}\times \overline{ \bigcap S^{w}_{\lambda_w}}^{T})_{\nu}.$
\end{proof}

\begin{conj}
T-fixed subschemes flatness imply flatness.
\end{conj}
We take $\lambda_w=-w_0 \lambda+w\lambda$, then
$\overline{\bigcap S^{w}_{\lambda_w}}=\overline{Y^{\lambda}}$. In this case the conjecture is proved to be true.
This flatness is mentioned in \cite{kamnitzer2016reducedness} remark 4.3 and will reduce the proof of reduceness of $\overline{Y^{\lambda}}$ to the case when $\lambda$ is $\varpi_i$ for each $i\in I$.

\section{Appendix}

%Let $\Gamma_j=W\cdot \varpi_{j}$ be the $W$ orbit of $\varpi_j.$ 
%We show that we can get $\Gamma_j$ by the following recursive way by applying certain simple reflection to $\varpi_j$.
%First Set $F_0=\{\varpi_j\}$.
%Let $F_{m+1}$ be the set of all elements that can be produced by certain simple reflections of elements in the previous step: $F_{m+1}$ consists all $s_i\gamma$ such that $\gamma \in F_m$, $i\in I$ and $ \langle\gamma,\alpha_i\rangle>0$.
%\begin{lemma}

%$\Gamma_{j}=\bigcup F_m$.
%\end{lemma}
%\begin{proof}

%From the definition of $F_m$, it suffices to show that:
%\begin{align}
%\textup{For any } \gamma\in F_m, \textup{ if } \langle \gamma, \alpha_j \rangle\leq 0 , \textup{ then } s_j\gamma\in \bigcup_{t\leq m} F_t.
%\end{align}
%Since this will imply that $F_{m+1}$ contains all $s_i\gamma$ for $i\in I$ and $\gamma\in F_m$.\\
For an expression $s_{i_{m}}\cdots s_{i_{1}}$ of an element w in W, we say it is j-admissible if $\langle \alpha_{i_a},s_{i_{a-1}}\cdots s_{i_{1}}\varpi_j\geq0$ for any $a\leq m$.
\begin{lemma}

For any element $w\in W$, any reduced expression of $w$ is j-admissible.(since we will fix an j, we will omit j and just say admissible).
\end{lemma}
\begin{proof}

Since we are in the ADE case, 
\begin{align}
s_i \varpi_i&=-\varpi_i+\sum_{h \text{ is adjacent to i }} \varpi_h. \\
s_i \varpi_h&= \varpi_h, \textup{ for $h\ne i$.}
\end{align}
We use induction on the length of $w$. 
Suppose lemma holds when $l(w)\leq m$.
Take  a reduced expression of $w \in W$ with length $m+1$ : 
$w=s_{i_{m+1}}\cdots s_{i_{1}}$.
Suppose this expression is not admissible,
we have 
$\langle \alpha_{i_{m+1}},s_{i_{m}}\cdots s_{i_{1}}\varpi_j\rangle \leq0$. 
Since $\langle \alpha_{i_{m+1}}, \varpi_j \rangle \geq0$, and by (6),(7) $$\langle \alpha_{i_{m+1}}, s_t \gamma \rangle \geq \langle \alpha_{i_{m+1}}, \gamma \rangle. $$
unless $t=i_{m+1}$, there must exists $k$ such that $i_k=i_{m+1}$.Let k be the biggest number such that  $i_{k}=i_{m+1}$.\\
In the case there is no element in the set $\{i_m,\cdots ,i_{k+1}\}$ is adjacent to $i_{m+1}$ in the Coxeter diagram, $s_{i_{m+1}}$ commutes with $s_{i_m}\cdots s_{i_{k+1}}$. Therefore
$s_{i_{m+1}} s_{i_{m}}\cdots s_{i_{k+1}} s_{i_{m+1}}=s_{i_{m+1}} s_{i_{m+1}}
s_{i_{m}}\cdots s_{i_{k+1}}=s_{i_{m}}\cdots s_{i_{k+1}}$
so the $w=s_{i_{m+1}}\cdots s_{i_{1}}$ is not reduced, contradiction.\\
In the case where for some $t$, $i_{t}$ is adjacent to $i_{m+1}$, we will show we can reduce to the case we have only one such t.
Suppose we have at least two elements $i_{t_1},i_{t_2}\cdots ,i_{t_h}$ such that they are all adjacent to $i_{m+1}$. 
Since $\langle \alpha_{i_{m+1}},s_{i_{m+1}}\cdots s_{i_{k}}\cdots s_{i_1}\varpi_j\rangle \leq0$ and $h>1$, by (6), (7), we must have some $i_{u_1}, i_{u_2}$ such that they are adjacent to $i_{m+1}$.
Since one point at most has 3 adjacent points we must have some $i_{u_x}=i_{t_y}.$
Let $p=i_{u_x}=i_{t_y}$. Using $s_p s_{i_{m+1}} s_p=s_{i_{m+1}} s_p s_{i_{m+1}}$ we can move $s_{i_{m+1}}$ in front of $s_{t_y}$ so the number h is reduced by 1.
We could do this procedure until h=1.
In this case we can rewrite the sequence before $s_{i_{k-1}}$ using the braid relation between $i_{m+1}$ and $i_t$: 
$$s_{i_m+1}\cdots s_{i_{t}} \cdots s_{i_k}=s_{i_m+1}\cdots s_{i_{t}} \cdots s_{{i_m+1}}= \cdots s_{i_{m+1}} s_{i_{t}} s_{i_{m+1}} \cdots=
\cdots s_{i_{t}} s_{i_{m+1}} s_{i_{t}}\cdots.$$
Set $\beta=s_{i_{k-1}}\cdots s_{i_{1}}\varpi_{j}$. 
By induction hypothesis, $s_{i_{m+1}} s_{i_{t}} s_{i_{m+1}} \cdots $and $s_{i_{t}} s_{i_{m+1}} s_{i_{t}}\cdots $ are admissible.
So $\langle \alpha_{i_t}, \beta \rangle \geq0$ and $\langle \alpha_{i_{m+1}}, \beta \rangle \geq0$. Again using (6) and (7) we have $\langle s_{i_{t}}   s_{i_{m+1}}\beta, \alpha_{i_{m+1}}\rangle\geq0$.
Then $\langle s_{i_{m+1}}\cdots s_{i_{1}}\varpi_j, \alpha_{i_{m+1}}\rangle
=\langle s_{i_{t}}   s_{i_{m+1}}\beta, \alpha_{i_{m+1}}\rangle\geq0$, contradicts with $s_{i_{m+1}}\cdots s_{i_{1}}$ is not admissible. 

\end{proof}

\begin{proof}[Proof of lemma 2]

Set $F_0=\{\varpi_j\}$.
Let $F_{m}$ be the set which contains all $w\varpi_j$, 
where $l(w)\leq m$.
We use induction. Suppose lemma 2 holds for $\gamma\in F_{m}$, we will prove lemma holds when $\gamma\in F_{m+1}.$
For any $\gamma=w\varpi_j \in F_{m+1}$, by lemma 1, w has a reduced admissible expression: $w=s_{i_{m+1}}\cdots s_{i_{1}}$. Denote $i_{m+1}$ by i and $\beta$ by  $s_{i_{m}}\cdots s_{i_{1}}\varpi_j$. So 
 $\gamma=s_{i}\beta$, $\beta \in F_{m}$.
 Since $s_{i_{m+1}}\cdots s_{i_{1}}$ is admissible, 
$\langle \beta, \alpha_i \rangle\geq 0$. Therefore $\langle \gamma, \alpha_i \rangle=
\langle s_i\beta, \alpha_i \rangle=-
\langle \beta, \alpha_i \rangle
\leq 0$ and we can apply prop 4.1 in \cite{MR2892443}.
Then $D_{\gamma}(M)=D_{s_{i}(s_{i}\gamma)}(M)=D_{s_{i}\gamma}(\Sigma_{i}M)$ , where $\Sigma_i$ is the reflection functor defined in section 2.2 in \cite{MR2892443}.\\
Let $A=\{j\ |\ j \text{ is adjacent to i, } j\in I\}$ and
 $M_{A}=\oplus_{s\in A} M_s$.
The $i^{th}$ component of $\Sigma_{i}M$
is the kernel of the map $\xi$ (Still see section 2.2 in \cite{MR2892443} for the definition of $\xi$) from $M_{A}$ to $M_{i}$. 
Since $\beta\in F_{m}$,  by induction hypothesis, we can apply this lemma to the case where $\gamma$ is taken to be $\beta$ and the module $M$ is $\Sigma_{i}M$.
Recall we denote by $I_\gamma^{+}$ the subset of I containing all i such that $\langle \gamma, \check{\alpha_i}\rangle$ is positive and by $I_\gamma^{-}$ containing all i $\langle \gamma, \check{\alpha_i}\rangle$ is negative.

Denote $A_{+}=\{j\ |\ j \text{ is adjacent to i, } j\in I_{\gamma}^{+}\}$ and $A_{-}=A\setminus A_{+}.$
For a multiset $S$, let $M_S=\oplus M_s^{m(s)}$. 
Regarding $I_{\gamma}^{-}$ as a multiset by setting $m(i)=\gamma_i^-$, we can rewrite  
$\oplus_{i\in I_{\gamma}^{-}}M_{i}^{\gamma^{-}}$ as $
M_{I^{-}_{\gamma}}$, similarly $\oplus_{i\in I_{\gamma}^{+}}M_{i}^{\gamma^{+}}$ as $
M_{I^{+}_{\gamma}}$.

Consider the case when $\langle \gamma, \alpha_i \rangle =-1$. 
We have  $I_{\beta}^{+}=I_{s_i\gamma}^{+}=(I_{\gamma}^{+}\setminus A_{+})\bigcup \{i\}$ and $I_{\beta}^{-}=I_{s_i\gamma}^{-}=(I_{\gamma}^{-}\setminus \{i\}) \bigcup A_{-}$ as multisets. Therefore
$D_{s_{i}\gamma}(\Sigma_{i}M)$ is the dimension of the kernel the natural map (which is $\phi_{\beta}$) from $M_{I^{+}_{\gamma}\setminus A_{+}}\oplus ker(M_{A}\xrightarrow[]{\xi} M_{i})$ to $M_{I^{-}_{\gamma} \setminus\{i\}}\oplus M_{A_{-}}$. This is equal to the dimension of the kernel of the natural map from $M_{I^{+}_{\gamma}\setminus A_{+}}\oplus ker(M_{A_{+}}\xrightarrow[]{\xi} M_{i})$ to
$M_{I^{+}_{\gamma}\setminus \{i\}}$, which is just $ker(M^{+}_{I_{\gamma}}\xrightarrow[]{\phi_{\gamma}} M_{I^{+}_{\gamma}})$. The case when $\langle \gamma, \alpha_i \rangle=-2$ is similar.

\end{proof}

\bibliographystyle{alpha}
\bibliography{sample.bib}

University of Massachusetts, Amherst, MA.\\
E-mail address: zdong@math.umass.edu
\end{document}